\documentclass[11pt]{amsart}
\usepackage{amssymb}

\makeatletter
\newcommand{\sumprime}{\if@display\sideset{}{'}\sum%
            \else\sum'\fi}
\makeatother

\begin{document}

\numberwithin{equation}{section}

\newtheorem{theorem}{Theorem}[section]
\newtheorem{proposition}[theorem]{Proposition}
\newtheorem{conjecture}[theorem]{Conjecture}
\def\theconjecture{\unskip}
\newtheorem{corollary}[theorem]{Corollary}
\newtheorem{lemma}[theorem]{Lemma}
\newtheorem{observation}[theorem]{Observation}
\newtheorem{definition}{Definition}
\numberwithin{definition}{section} 
\newtheorem{remark}{Remark}
\def\theremark{\unskip}
\newtheorem{question}{Question}
\def\thequestion{\unskip}
\newtheorem{example}{Example}
\def\theexample{\unskip}
\newtheorem{problem}{Problem}

\def\vvv{\ensuremath{\mid\!\mid\!\mid}}
\def\intprod{\mathbin{\lr54}}
\def\reals{{\mathbb R}}
\def\integers{{\mathbb Z}}
\def\N{{\mathbb N}}
\def\complex{{\mathbb C}\/}
\def\dist{\operatorname{dist}\,}
\def\spec{\operatorname{spec}\,}
\def\interior{\operatorname{int}\,}
\def\trace{\operatorname{tr}\,}
\def\cl{\operatorname{cl}\,}
\def\essspec{\operatorname{esspec}\,}
\def\range{\operatorname{\mathcal R}\,}
\def\kernel{\operatorname{\mathcal N}\,}
\def\dom{\operatorname{Dom}\,}
\def\linearspan{\operatorname{span}\,}
\def\lip{\operatorname{Lip}\,}
\def\sgn{\operatorname{sgn}\,}
\def\Z{ {\mathbb Z} }
\def\e{\varepsilon}
\def\p{\partial}
\def\rp{{ ^{-1} }}
\def\Re{\operatorname{Re\,} }
\def\Im{\operatorname{Im\,} }
\def\dbarb{\bar\partial_b}
\def\eps{\varepsilon}
\def\O{\Omega}
\def\Lip{\operatorname{Lip\,}}

\def\Hs{{\mathcal H}}
\def\E{{\mathcal E}}
\def\scriptu{{\mathcal U}}
\def\scriptr{{\mathcal R}}
\def\scripta{{\mathcal A}}
\def\scriptc{{\mathcal C}}
\def\scriptd{{\mathcal D}}
\def\scripti{{\mathcal I}}
\def\scriptk{{\mathcal K}}
\def\scripth{{\mathcal H}}
\def\scriptm{{\mathcal M}}
\def\scriptn{{\mathcal N}}
\def\scripte{{\mathcal E}}
\def\scriptt{{\mathcal T}}
\def\scriptr{{\mathcal R}}
\def\scripts{{\mathcal S}}
\def\scriptb{{\mathcal B}}
\def\scriptf{{\mathcal F}}
\def\scriptg{{\mathcal G}}
\def\scriptl{{\mathcal L}}
\def\scripto{{\mathfrak o}}
\def\scriptv{{\mathcal V}}
\def\frakg{{\mathfrak g}}
\def\frakG{{\mathfrak G}}

\def\ov{\overline}

\thanks{Research supported by the Key Program of NSFC No. 11031008.}

\address{Department of Mathematics, Tongji University, Shanghai, 200092, China}
 \email{boychen@tongji.edu.cn}

 \address{Department of Mathematics, Tongji University, Shanghai, 200092, China}
 \email{99jujiewu@tongji.edu.cn}

 \address{Department of Mathematics, Tongji University, Shanghai, 200092, China}
 \email{1113xuwang@tongji.edu.cn}

\title{Ohsawa-Takegoshi type theorem and extension of plurisubharmonic functions}
\author{Bo-Yong Chen}
\author{Jujie Wu}
\author{Xu Wang}
\date{}
\maketitle

\bigskip

\begin{abstract}

We prove a Thullen type extension theorem of plurisubharmonic functions across a closed complete pluripolar set, which generalizes a theorem of Siu. Our approach depends on an Ohsawa-Takegoshi type extension theorem for a single point in a bounded complete K\"ahler domain, which is of independent interest.

\bigskip


\smallskip

\end{abstract}

\section{Introduction}

Removable singularities of analytic objects, e.g., holomorphic or plurisubharmonic (psh) functions, or more generally, holomorphic maps and closed positive currents, are of classical interest and have been deeply studied. A good reference is Carleson \cite{Carleson67}, Siu \cite{SiuTechExten}, or Harvey-Polking \cite{HarveyPolking75}. In this paper, we are going to focus on the following Thullen type extension theorem of Siu:

\begin{theorem}[cf. \cite{SiuLelong}]
Let $\Omega$ be an open subset of ${\mathbb C}^n$, and $V$ a subvariety of $\Omega$.
 Suppose $V$ is of codimension $\ge 1$, and $G$ is an open subset of $\Omega$ which intersects every branch of $V$ of codimension 1. Then every psh function $\varphi$ on $(\Omega\backslash V)\cup G$ can be extended to a psh function on $\Omega$.
\end{theorem}

Since the result is local, one may reduce to the special case when $\Omega$ is the unit polydisc ${\mathbb D}^{n+1}$ in ${\mathbb C}^{n+1}$, $V$ is the hyperplane defined by $z_{n+1}=0$, and $G={\mathbb D}_r^n\times {\mathbb D}$ with some $r<1$. Siu's proof depends on H\"ormander's $L^2$ estimates of the $\bar{\partial}-$operator, which is also the main tool of that famous paper \cite{SiuLelong}. Harvey-Polking \cite{HarveyPolking75} found a different (and more elementary) approach of Siu's theorem, under the weaker condition that $\varphi(z',z_{n+1})$ extends to a subharmonic function on ${\mathbb D}$ for any $z'$ lying in some nonpluripolar subset of ${\mathbb D}^n$.

We shall prove the following

\begin{theorem}\label{th:pshextension}
Let $E$ be a closed complete pluripolar subset of\/ ${\mathbb D}^{n+1}$ such that $E\subset {\mathbb D}^{n}\times {\mathbb D}_r$ for some $0<r<1$. Let $\varphi\in PSH({\mathbb D}^{n+1}\backslash E)$.  Suppose
    there is a set $A\subset {\mathbb D}^{n}$ of positive Lebesgue measure such that for each $z'\in A$, the restriction of $\varphi$ to $(\{z'\}\times {\mathbb D})\backslash E$ is locally bounded from above near $(\{z'\}\times {\mathbb D})\cap E$.  Then $\varphi$
 can be extended to a psh function on ${\mathbb D}^{n+1}$.
 \end{theorem}

Our analysis also depends on $L^2-$theory for the $\bar{\partial}-$operator. More precisely, we use the following Ohsawa-Takegoshi type extension theorem:

\begin{theorem}
Let $\Omega$ be a bounded complete K\"ahler domain in ${\mathbb C}^n$, and $\varphi\in PSH(\Omega)$. For any $a\in \Omega$ and any complex number $c$ with $|c|^2\le e^{\varphi(a)}$, there exists a holomorphic function $f$ on $\Omega$ such that $f(a)=c$ and
$$
\int_\Omega |f|^2 e^{-\varphi}\le {\rm const.}
$$
where the constant depends only on $n$ and the diameter of $\Omega$.
\end{theorem}

   If $\Omega$ is pseudoconvex, the above result is only a rather special case of the original Ohsawa-Takegoshi extension theorem, which has, however, a significant application: Demailly's approximation theory \cite{Demailly92}.  Although the key ingredient in proving Theorem 1.3 is still the a priori inequality for the (twisted) $\bar{\partial}-$operator due to Ohsawa-Takegoshi \cite{OhsawaTakegoshi87}, a crucial difficulty arises when $\Omega$ is merely complete K\"ahler, since one can neither find in general for each $\varphi\in PSH(\Omega)$ a sequence of smooth psh functions defined on $\Omega$ decreasing to $\varphi$, nor a sequence of complete K\"ahler subdomains increasing to $\Omega$, so that the classical approximating procedure of H\"ormander breaks down. To overcome such a difficulty, we propose a new (and even more natural) approach by solving at first the (twisted) Laplace equations (with Dirichlet boundary conditions) on a sequence of smooth subdomains increasing to $\Omega$, then use these solutions to approximate the needed solution of the (twisted) $\bar{\partial}-$equation on $\Omega$ with $L^2$ estimate. Such an idea is motivated by the work of Bando \cite{Bando}.

At first sight, there seem to be no obstructions for extending Theorem 1.3 to arbitrary bounded domains. However, we shall show that the conclusion of Theorem 1.3 fails for each generalized annulus of dimension $\ge 2$ (see Proposition 6.1).

 Combining Theorem 1.2 with Siu's method (cf. \cite{SiuLelong}), we may prove the following

\begin{theorem}
Let $E$ be a compact polar set in the unit disc ${\mathbb D}$ and
$$
\Omega=({\mathbb D}^n_\delta\times {\mathbb D})\cup \left({\mathbb D}^n\times ({\mathbb D}\backslash E)\right)
$$
for $ 0<\delta<1$. Then every closed positive $(1,1)-$current on $\Omega$ can be extended to a closed positive $(1,1)-$current on ${\mathbb D}^{n+1}$.
\end{theorem}

It seems that the method in proving Theorem 1.3 may be used to develop a fairly general $L^2$ theory of the $\bar{\partial}-$operator on complete K\"ahler manifolds, built on Demailly's approximation theory.  We shall treat it in a future paper.

 \section{Polar sets and extension of analytic objects with $L^2$ conditions}

Recall that a set $E\subset {\mathbb R}^n$ is called\/ {\it polar}\/ if every point in $E$ has a neighborhood $U$ where there is a subharmonic  function $\psi$ not identically equals $-\infty$, such that $\psi=-\infty$ on $E\cap U$. It is known that polarity is invariant under diffeomorphisms, thus makes sense on manifolds.  Replacing subharmonic functions by psh functions, we may define pluripolar sets analogously.

For the proof of Theorem 1.2, we need the following classical result:

 \begin{theorem} [cf. \cite{Conway95}]
 Let $\Omega$ be a domain in ${\mathbb C}$ and $E$ a closed subset of $\Omega$. Then every $f\in L^2_{\rm loc}(\Omega)\cap {\mathcal O}(\Omega\backslash E)$ can be extended to a holomorphic function on $\Omega$ if and only if $E$ is polar.
 \end{theorem}

 Based on this fact, Siciak \cite{Siciak82} was able to show that a closed pluripolar subset $E$ of a domain $\Omega\subset {\mathbb C}^n$ is always removable for $f\in L^2_{\rm loc}(\Omega)\cap {\mathcal O}(\Omega\backslash E)$. In this section, we would like to give a completely different approach for the "if" part of Theorem 2.1, as well as far reaching generalizations and applications.  Our start point is the following elementary gradient estimate for subharmonic functions, which seems useful for other purposes.

Consider a domain $\Omega\subset {\mathbb R}^n$ and a subharmonic function $\psi$ on $\Omega$ which is not identically $-\infty$. Then $\Delta \psi$ defines a positive Radon measure on $\Omega$. Set $E=\psi^{-1}(-\infty)$ and
$$
\Omega_R=\{z\in \Omega:\psi<-R\}
$$
 for $R>0$. We claim the following

\begin{lemma}
Suppose furthermore $E$ is\/ {\rm compact},  $\Omega_1\subset \subset \Omega$ and $\psi\in C^2(\Omega\backslash E)$. Then we have
$$
\int_{\Omega_{2}\backslash E} |\nabla\psi|^2/\psi^2 \le \log 2 \int_{\Omega_1}\Delta \psi.
$$
\end{lemma}

\begin{proof}
For any positive integer $\nu$, we take a convex increasing function $\lambda_\nu$ on ${\mathbb R}$ such that $\lambda_\nu(t)=t$ for $t\ge -\nu$ and $\lambda_\nu(t)=-\nu-1/2$ for $t\le -\nu-1$. Put
 $\psi_\nu=\lambda_\nu\circ \psi$. Clearly, $\psi_\nu\downarrow \psi$ as $\nu\rightarrow \infty$. Let $\chi:{\mathbb R}\rightarrow [0,1]$ be a cut-off function such that $\chi|_{(-\infty,-\log 2)}=1$, $\chi|_{(0,\infty)}=0$, $\chi'\le 0$ and $\chi'|_{(-\log 2,0)}= -1/\log 2$. For positive integers $\nu,\mu$, we set
$$
\lambda_{\nu,\mu}=\chi(-\log(-\psi_\nu)+\mu\log 2).
$$
It follows from Green's theorem that
\begin{eqnarray*}
\int \lambda_{\nu,\mu} \Delta \psi_\nu & = & -\int \nabla\lambda_{\nu,\mu}\cdot \nabla\psi_\nu =\int \chi'(\cdot) \frac{|\nabla\psi_\nu|^2}{\psi_\nu}\\
& \ge & \frac1{\log 2}\int_{\psi_\nu^{-1}((-2^{\mu+1},-2^\mu))}\frac{|\nabla\psi_\nu|^2}{-\psi_\nu}\\
& \ge & \frac{2^\mu}{\log 2}\int_{\psi_\nu^{-1}((-2^{\mu+1},-2^\mu))}\frac{|\nabla\psi_\nu|^2}{\psi^2_\nu}
\end{eqnarray*}
so that
$$
\int_{\psi_\nu^{-1}((-2^{\mu+1},-2^\mu))}\frac{|\nabla\psi_\nu|^2}{\psi_\nu^2}\le \frac{\log 2}{2^\mu} \int_{\Omega_1} \Delta \psi_\nu.
$$
Hence
$$
\int_{\Omega_{2}\backslash E} |\nabla\psi_\nu|^2/\psi_\nu^2 =\sum_{\mu=1}^\infty \int_{\psi_\nu^{-1}((-2^{\mu+1},-2^\mu))}\frac{|\nabla\psi_\nu|^2}{\psi_\nu^2}\le \log 2\int_{\Omega_1} \Delta \psi_\nu.
$$
Letting $\nu\rightarrow \infty$, we immediately get the desired inequality.

\end{proof}

\begin{theorem}
Let $E$ be a closed polar subset of a domain $\Omega\subset {\mathbb R}^n$, and
$$
P\left(x,D^{(1)}\right)=a_0(x)+a_1(x)\partial/\partial x_1+\cdots + a_n(x)\partial/\partial x_n
$$
where $a_0\in L^2_{\rm loc}(\Omega)$ and $a_k\in C^1(\Omega)$ for $1\le k\le n$. Suppose we have $f\in L^2_{\rm loc}(\Omega)$ and $v\in L^1_{\rm loc}(\Omega)$ so that  $P\left(x,D^{(1)}\right) f=v$  on $\Omega\backslash E$ in the sense of distributions. Then $P\left(x,D^{(1)}\right) f=v$ in $\Omega$.
\end{theorem}

\begin{proof}
 The statement is local on $\Omega$, so we may work on a small neighborhood $U$ of a given point $z^0\in E$. Take an open set $V\subset\subset \Omega$ such that $U\subset\subset V$.  Since $E\cap \overline{V}$ is a compact polar set, there is a positive measure $\lambda$ such that the corresponding potential
$$
p_\lambda(x):=\left\{
\begin{array}{ll}
-\int_{{\mathbb R}^n}\frac1{|x-y|^{n-2}}d\lambda(y) & {\rm if\ } n>2\\
\int_{{\mathbb R}^n}\log |x-y|d\lambda(y) & {\rm if\ } n=2
\end{array}
\right.
$$
enjoys the following properties: $p_\lambda(x)=-\infty$ for $x\in E\cap \overline{V}$ and $p_\lambda(x)>-\infty$ outside $E\cap \overline{V}$, in view of Evans' theorem (see e.g., \cite{Landkof}, Theorem 3.1). It is also well-known that $p_\lambda$ is subharmonic in ${\mathbb R}^n$ and harmonic outside $E\cap \overline{V}$. Replacing $p_\lambda$ by $p_\lambda-C$ where $C$ is a sufficiently large constant, we may assume $p_\lambda<0$ in $U$ so that
$$
\int_{U \backslash E} |\nabla p_\lambda|^2/p_\lambda^2<\infty,
$$
in view of Lemma 2.2. It suffices to show
\begin{equation}
\int_U \left (P\left(x,D^{(1)}\right)f\right) g=\int_U v g
\end{equation}
for each $g\in C^\infty_0(U)$.
For any $R>0$, we define
$$
\chi_R(z)=\kappa(\log(-p_\lambda)-\log R)
$$
 where $\kappa$ is a cut-off function on ${\mathbb R}$ such that $\kappa|_{(-\infty,0)}=1$ and $\kappa|_{[1,\infty)}=0$. It is easy to verify that $\chi_R\,g\in C^\infty_0(U\backslash E)$ provided $R$ sufficiently large. Thus
 \begin{eqnarray}
 \int_U v \chi_R g & = & \int_U \left(P\left(x,D^{(1)}\right)f\right) \chi_R g\nonumber\\
  & = & \int_U \chi_R a_0 f g+\sum_{k=1}^n \int_U a_k \frac{\partial f}{\partial x_k} \chi_R g\nonumber \\
 & = &  \int_U \chi_R a_0 f g-\sum_{k=1}^n \int_U f \frac{\partial}{\partial x_k}( a_k \chi_R g).
 \end{eqnarray}
 Since
 $$
 \int_U f \frac{\partial}{\partial x_k}( a_k \chi_R g)= \int_U \chi_R f \frac{\partial}{\partial x_k}( a_k g)+\int_U a_k f g \frac{\partial}{\partial x_k} \chi_R
 $$
 while
 $$
 \left|\int_U a_k f g \frac{\partial}{\partial x_k} \chi_R\right|^2\le \int_{{\rm supp\,}\kappa'(\cdot)} |a_k fg|^2 \int_U |\kappa'(\cdot)|^2 |\nabla p_\lambda|^2/p_\lambda^2\rightarrow 0
 $$
 as $R\rightarrow \infty$, it follows from (2.2) that
 $$
 \int_U a_0 f g-\sum_{k=1}^n \int_U f \frac{\partial}{\partial x_k}( a_k g)=\int_U vg,
 $$
 from which (2.1) immediately follows.
\end{proof}

Viewing the $\bar{\partial}-$equation as a finite number of first order partial differential equations, we may extend a result of Demailly (cf. \cite{Demailly82}, Lemma 6.9) as follows

 \begin{corollary}\label{th:dbarextension}
Let $\Omega$ be a domain in ${\mathbb C}^n$ and $E\subset \Omega$ a closed polar set. Let $u$ be a $(p,q-1)-$form with $L^2_{{\rm loc}}$ coefficients  and $v$ a $(p,q)-$form with $L^1_{{\rm loc}}$ coefficients such that $\bar{\partial}u=v$ on $\Omega\backslash E$ in the sense of distributions. Then $\bar{\partial}u=v$ on $\Omega$.
\end{corollary}

In particular, a closed\/ {\it polar}\/ subset $E$ of a domain $\Omega\subset {\mathbb C}^n$ is always removable for $f\in L^2_{\rm loc}(\Omega)\cap {\mathcal O}(\Omega\backslash E)$, in view of Weyl's lemma.

\begin{remark}
It follows from Diederich-Pflug \cite{DiederichPflug} that if\/ $\Omega\subset {\mathbb C}^n$ is a bounded complete K\"ahler domain and $E:=\overline{\Omega}^\circ\backslash \Omega$ is removable for $f\in L^2_{\rm loc}(\Omega)\cap {\mathcal O}(\Omega\backslash E)$, then $\overline{\Omega}^\circ$ is Stein. Thus if $E$ is polar, then $\overline{\Omega}^\circ$ is Stein.
\end{remark}

Another consequence of Theorem 2.3 is the following Bando type theorem (compare \cite{Bando}, Theorem 1)

\begin{corollary}
Let $E$ be a closed polar subset of the unit ball ${\mathbb B}\subset {\mathbb C}^n$, and $(L,h)$ a holomorphic Hermitian line bundle over ${\mathbb B}\backslash E$. If the curvature $\Theta_h$ of $(L,h)$ is square integrable, then $L$ is flat, i.e., it comes from a representation of $\pi_1({\mathbb B}\backslash E)$.
\end{corollary}

In particular, if ${\mathbb B}\backslash E$ is simply connected, then $L$ can be extended to a holomorphic line bundle on ${\mathbb B}$ and hence is globally trivial over ${\mathbb B}\backslash E$.

\begin{proof}
Since $\Theta_h$ is a square integrable $d-$closed $(1,1)-$form on ${\mathbb B}\backslash E$, so it can be extended to a $d-$closed $(1,1)-$current $u$ on ${\mathbb B}$ in view of Theorem 2.3. Thus there exists a real-valued function $\varphi$ on ${\mathbb B}$ such that $i\partial\bar{\partial}\varphi=u$ and $\varphi$ is smooth on ${\mathbb B}\backslash E$. If we put $h'=he^\varphi$, then $\Theta_{h'}=0$. It follows that $L$ is flat.
\end{proof}

Similar as Theorem 2.3, we may prove the following

\begin{proposition}
Let $\Omega$ be a domain in ${\mathbb C}^n$ and $E\subset \Omega$ a closed polar set. Then every $\varphi\in W^{1,2}_{\rm loc}(\Omega)\cap PSH(\Omega\backslash E)$ can be extended to a psh function on $\Omega$.
\end{proposition}

\begin{proof} Let $U$ and $\chi_R$ be as in the proof of Theorem 2.3. For each unit vector $\zeta\in {\mathbb C}^n$ and non-negative test function $g\in C^\infty_0(U)$, we have
\begin{eqnarray}
  \int \chi_R \varphi \sum_{j,k} \frac{\partial^2 g }{\partial z_j\partial\bar{z}_k}\zeta_j\bar{\zeta}_k & = &
   \int \varphi \sum_{j,k} \frac{\partial^2 (\chi_R\,g) }{\partial z_j\partial\bar{z}_k}\zeta_j\bar{\zeta}_k-2{\rm Re}\int \varphi \sum_{j,k} \frac{\partial \chi_R}{\partial z_j}\frac{\partial g}{\bar{z}_k}\zeta_j\bar{\zeta}_k \nonumber\\
   && - \int \varphi g \sum_{j,k} \frac{\partial^2 \chi_R }{\partial z_j\partial\bar{z}_k}\zeta_j\bar{\zeta}_k.
  \end{eqnarray}
  An integration by parts yields
  $$
   - \int \varphi g \sum_{j,k} \frac{\partial^2 \chi_R }{\partial z_j\partial\bar{z}_k}\zeta_j\bar{\zeta}_k=\int \varphi \sum_{j,k} \frac{\partial g}{\partial z_j}\frac{\partial  \chi_R}{\bar{z}_k}\zeta_j\bar{\zeta}_k+\int g \sum_{j,k} \frac{\partial \varphi}{\partial z_j}\frac{\partial  \chi_R}{\bar{z}_k}\zeta_j\bar{\zeta}_k.
  $$
Since the first integral in RHS of (2.3) is nonnegative, and the remaining integrals are dominated by
$$
\|\varphi\|_{W^{1,2}(U)}\|\nabla\chi_R\|_{L^2(U)},
$$
so we have
$$
\int  \varphi \sum_{j,k} \frac{\partial^2 g }{\partial z_j\partial\bar{z}_k}\zeta_j\bar{\zeta}_k \ge 0
$$
by letting $R\rightarrow \infty$ in (2.3). The desired psh extension of $\varphi$ to $U$ is
$$
\hat{\varphi}=\lim_{\varepsilon\rightarrow 0} (\varphi\ast \kappa_\varepsilon)
$$
where $\kappa_\varepsilon$ is a standard smoothing kernel.

\end{proof}

\section{Proof of Theorem 1.3}
Consider first a bounded domain $\Omega\subset {\mathbb C}^n$ with $C^2-$boundary. Let $g$ be a K\"ahler metric and $\varphi$ a $C^2$ real function both defined in a neighborhood of $\overline{\Omega}$. Let $D_{(p,q)}(\Omega)$ denote the set of $C^\infty-$smooth  $(p,q)-$forms which are compactly supported in $\Omega$. For $u,v\in D_{(p,q)}(\Omega)$, we define the inner product as follows
$$
(u,v)_\varphi=\int_\Omega u\wedge \bar{\ast} v\,e^{-\varphi}
$$
where $\ast$ is Hodge's star operator with respect to $g$. Let $\bar{\partial}^\ast_\varphi$ denote the formal adjoint of $\bar{\partial}$ corresponding to $(\cdot,\cdot)_\varphi$.

Suppose $\eta,\lambda$ are two $C^\infty$ positive functions on $\overline{\Omega}$. Then we have the following fundamental a priori inequality due to Ohsawa-Takegoshi \cite{OhsawaTakegoshi87} (see also \cite{Ohsawa95}, \cite{DemaillyBook10}):
\begin{equation}
\|(\sqrt{\eta}+\sqrt{\lambda})\bar{\partial}^\ast_\varphi u\|^2_\varphi+\|\sqrt{\eta}\bar{\partial}u\|^2_\varphi\ge ([\eta i\partial\bar{\partial}\varphi-i\partial\bar{\partial}\eta-\lambda^{-1}i\partial\eta\wedge \bar{\partial}\eta,\Lambda]u,u)_\varphi
\end{equation}
for $u\in D_{(n,1)}(\Omega)$. Suppose furthermore there exists a continuous positive $(1,1)-$form $\Theta$ on $\overline{\Omega}$ such that
$$
\eta i\partial\bar{\partial}\varphi-i\partial\bar{\partial}\eta-\lambda^{-1}i\partial\eta\wedge \bar{\partial}\eta\ge \Theta.
$$
Put $\rho=(\sqrt{\eta}+\sqrt{\lambda})^2$ and define the twisted complex Laplacian by
$$
\Box_\rho:=\bar{\partial}(\rho\cdot \bar{\partial}^\ast_\varphi)+\bar{\partial}^\ast_\varphi(\rho\cdot\bar{\partial}).
$$
Clearly, we have
$$
(\Box_\rho u,v)_\varphi=(u,\Box_\rho v)_\varphi
$$
for $u,v\in D_{(n,1)}(\Omega)$ and
\begin{equation}
(\Box_\rho u,u)_\varphi=\|\sqrt{\rho} \bar{\partial}^\ast_\varphi u\|^2_\varphi+\|\sqrt{\rho}\bar{\partial}u\|^2_\varphi\ge ([\Theta,\Lambda]u,u)_\varphi\ge c\|u\|_\varphi^2
\end{equation}
for suitable constant $c>0$.  Let $H\subset L^2_{(n,1)}(\Omega,g,\varphi)$ be the completion of $D_{(n,1)}(\Omega)$ under the norm
$$
\|\sqrt{\rho} \bar{\partial}^\ast_\varphi u\|_\varphi+\|\sqrt{\rho}\bar{\partial}u\|_\varphi.
$$
    In view of the Friedrichs Extension Theorem (see e.g., Folland-Kohn \cite{FollandKohn}, p.\,14\,), $\Box_\rho$ admits a unique self-adjoint extension, which is still denoted by $\Box_\rho$ for the sake of simplicity, with ${\rm Dom\,}\Box_\rho\subset H$ satisfying
 $$
 (\Box_\rho u,u)_\varphi=\|\sqrt{\rho} \bar{\partial}^\ast_\varphi u\|^2_\varphi+\|\sqrt{\rho}\bar{\partial}u\|^2_\varphi
 $$
 for all $u\in  {\rm Dom\,}\Box_\rho$, such that the following result holds:

\begin{lemma}
For any $v\in L^2_{(n,1)}(\Omega,g,\varphi)$, there is a unique $w\in H$ such that $\Box_\rho w=v$ and
$$
\max\left\{([\Theta,\Lambda]w,w)_\varphi, \|\sqrt{\rho} \bar{\partial}^\ast_\varphi w\|^2_\varphi,\|\sqrt{\rho}\bar{\partial}w\|^2_\varphi\right\} \le ([\Theta,\Lambda]^{-1}v,v)_\varphi.
$$
\end{lemma}

\begin{proof}
The first conclusion follows directly from \cite{FollandKohn}, p.\,14. For the $L^2-$estimates,  we infer from (3.2) that
$$
([\Theta,\Lambda]w,w)_\varphi^2\le (v,w)_\varphi^2\le ([\Theta,\Lambda]^{-1}v,v)_\varphi\cdot ([\Theta,\Lambda]w,w)_\varphi,
$$
so that $([\Theta,\Lambda]w,w)_\varphi\le ([\Theta,\Lambda]^{-1}v,v)_\varphi$. Similarly,
$$
\max\left\{\|\sqrt{\rho} \bar{\partial}^\ast_\varphi w\|^2_\varphi,\|\sqrt{\rho}\bar{\partial}w\|^2_\varphi\right\}\le (v,w)_\varphi\le ([\Theta,\Lambda]^{-1}v,v)_\varphi.
$$
\end{proof}

Now suppose $\Omega\subset {\mathbb C}^n$ is a bounded domain with a complete K\"ahler metric $g$. Let $\varphi\in PSH(\Omega)$ and $\eta,\lambda$ two positive $C^\infty$ functions on $\overline{\Omega}$.

\begin{proposition}
Suppose there is a continuous positive $(1,1)-$form $\Theta$ on $\overline{\Omega}$ such that
$$
\eta i\partial\bar{\partial}\varphi-i\partial\bar{\partial}\eta-\lambda^{-1}i\partial\eta\wedge \bar{\partial}\eta\ge \Theta
$$
holds in the sense of distributions. Then for any smooth $\bar{\partial}-$closed $(n,1)-$form $v$ such that
$$
([\Theta,\Lambda]^{-1}v,v)_\varphi<\infty,
$$
there is a solution $u\in L^2_{(n,0)}(\Omega,{\rm loc})$ of the equation $\bar{\partial}u=v$ such that
$$
\left|\int_\Omega \frac{u\wedge \bar{u}}{(\sqrt{\eta}+\sqrt{\lambda})^2}e^{-\varphi}\right|\le ([\Theta,\Lambda]^{-1}v,v)_\varphi.
$$
\end{proposition}

\begin{proof}
Take a sequence of domains $\Omega_j\subset \Omega$ with $C^\infty-$boundaries such that $\overline{\Omega}_j\subset \Omega_{j+1}$, $\Omega=\cup \Omega_j$, and a sequence of strictly psh functions $\varphi_j$ on $\Omega_{j+1}$ such that $\varphi_j\downarrow \varphi$ on $\Omega$ as $j\rightarrow \infty$. Applying Lemma 3.1 to $(\Omega_j,g,\varphi_j,\eta,\lambda,\rho)$, we get a solution $w_j$ of the equation $\Box_\rho^{(j)}w=v$ such that
$$
\max\left\{([\Theta,\Lambda]w_j,w_j)_{\varphi_j}, \|\sqrt{\rho} \bar{\partial}^\ast_{\varphi_j} w_j\|^2_{\varphi_j},\|\sqrt{\rho}\bar{\partial}w_j\|^2_{\varphi_j}\right\} \le ([\Theta,\Lambda]^{-1}v,v)_{\varphi_j}\le ([\Theta,\Lambda]^{-1}v,v)_{\varphi}.
$$
Furthermore, $w_j$ is smooth on $\Omega_j$ in view of the theory of elliptic operators (see \cite{GilbergTrudinger}). Put $u_j=\rho  \bar{\partial}^\ast_{\varphi_j} w_j$. Since $\|\rho^{-1/2}u_j\|_{\varphi_j}^2\le ([\Theta,\Lambda]^{-1}v,v)_{\varphi}$, there is a subsequence, which is still denoted by $u_j$, such that $\rho^{-1/2} u_j\rightarrow \rho^{-1/2}u$ weakly on $\Omega$ together with estimate
$$
\left|\int_\Omega \rho^{-1} u\wedge \bar{u}\,e^{-\varphi}\right|\le ([\Theta,\Lambda]^{-1}v,v)_{\varphi}.
$$
Thus it suffices to show that $\bar{\partial}u=v$ in the sense of distributions. Since
\begin{equation}
v=\Box_\rho^{(j)} w_j=\bar{\partial} u_j+\bar{\partial}^\ast_{\varphi_j} (\rho \bar{\partial} w_j),
\end{equation}
we only need to show that $\bar{\partial}^\ast_{\varphi_j} (\rho \bar{\partial} w_j)\rightarrow 0$ in the sense of distributions.

Without loss of generality, we may assume $0\in \Omega$. Let ${\rm dist}_g(0,z)$ denote the distance between $0,z$, w.r.t. $g$. Replacing the distance function by a smoothing of it, we may also assume, without loss of generality, that it is smooth. Let $\chi:{\mathbb R}\rightarrow [0,1]$ be a cut-off function satisfying $\chi|_{(-\infty,1/2)}=1$ and $\chi|_{(1,\infty)}=0$. For any $\varepsilon>0$, we define $\kappa_\varepsilon(z)=\chi(\varepsilon\cdot {\rm dist}_g(0,z))$. Clearly, ${\rm supp\,}\kappa_\varepsilon\subset \Omega_j$ provided $j$ sufficiently large, for $g$ is complete. Since $\bar{\partial} \bar{\partial}^\ast_{\varphi_j} (\rho \bar{\partial} w_j)=0$ holds in the sense of distributions on $\Omega_j$ in view of (3.3), so we have
\begin{eqnarray*}
0 & = & (\bar{\partial} \bar{\partial}^\ast_{\varphi_j} (\rho \bar{\partial} w_j),\kappa_\varepsilon^2\rho \bar{\partial} w_j)_{\varphi_j}
        =  \|\kappa_\varepsilon \bar{\partial}^\ast_{\varphi_j} (\rho \bar{\partial} w_j)\|^2_{\varphi_j}-2 ( \bar{\partial}^\ast_{\varphi_j} (\rho \bar{\partial} w_j),\kappa_\varepsilon \rho\bar{\partial} \kappa_\varepsilon \lrcorner\,  \bar{\partial} w_j)_{\varphi_j}
\end{eqnarray*}
where $"\lrcorner"$ is the contraction operator. It follows from the Schwarz inequality that
\begin{eqnarray*}
\|\kappa_\varepsilon \bar{\partial}^\ast_{\varphi_j} (\rho \bar{\partial} w_j)\|_{\varphi_j} & \le & 2M^{1/2}\varepsilon\sup|\chi'| \|\sqrt{\rho} \bar{\partial}w_j\|_{\varphi_j}\ \ \ \ (M:=\sup \{\rho(z):z\in \Omega\})\\
& \le & {\rm const.}\varepsilon\, ([\Theta,\Lambda]^{-1}v,v)_\varphi^{1/2}.
\end{eqnarray*}
Given $f\in D_{(n,1)}(\Omega)$, we may take $\varepsilon\ll 1$ and $j\ge j_0 \gg 1$ such that ${\rm supp\,}f\subset \Omega_j$ and $\kappa_\varepsilon=1$ on ${\rm supp\,}f$. It follows that
$$
|(\bar{\partial}^\ast_{\varphi_j} (\rho \bar{\partial} w_j),f)_{\varphi_{j_0}}|=|(\kappa_\varepsilon\bar{\partial}^\ast_{\varphi_j} (\rho \bar{\partial} w_j),f)_{\varphi_{j_0}}|\le \|\kappa_\varepsilon \bar{\partial}^\ast_{\varphi_j} (\rho \bar{\partial} w_j)\|_{\varphi_j} \|f\|_{\varphi_{j_0}}\rightarrow 0
$$
as $j\rightarrow \infty$ and $\varepsilon\rightarrow 0$, so that $\bar{\partial}^\ast_{\varphi_j} (\rho \bar{\partial} w_j)\rightarrow 0$ in the sense of distributions.
\end{proof}

\begin{remark}
In fact, the artificial hypothesis that $v$ is smooth can be removed. To see this, simply take a sequence of smooth $(n,1)-$forms $v_j$ converge weakly to $v$. Although $v_j$ is not $\bar{\partial}-$closed, $\bar{\partial} v_j$ converges weakly to $0$. Thus the previous argument still works.
\end{remark}

To proceed the proof, we also need the following simple observation:

\begin{lemma}
Let $B_\varepsilon\subset {\mathbb C}^n$ denote the ball with center $0$ and radius $\varepsilon$. Let $\varphi\in PSH(B_\varepsilon)$. For any complex number $c$ with $|c|^2\le e^{\varphi(0)}$, there exists a holomorphic function $f$ on $B_\varepsilon$ such that $f(0)=c$ and
$$
\varepsilon^{-2n}\int_{B_\varepsilon} |f|^2 e^{-\varphi}\le {\rm const}_n.
$$
\end{lemma}

\begin{proof}
Put $\varphi_\varepsilon(z)=\varphi(\varepsilon z)$, $z\in {\mathbb B}:$ the unit ball. In view of the Ohsawa-Takegoshi extension theorem, there is a holomorphic function $h$ on ${\mathbb B}$ such that $h(0)=c$ and
$$
\int_{\mathbb B} |h|^2 e^{-\varphi_\varepsilon}\le {\rm const}_n |c|^2e^{-\varphi_\varepsilon(0)}\le {\rm const}_n
$$
(compare \cite{Demailly92}). Put $f(z)=h(z/\varepsilon)$ for $z\in B_\varepsilon$. Then $f\in {\mathcal O}(B_\varepsilon)$, $f(0)=c$ and
$$
\varepsilon^{-2n} \int_{B_\varepsilon} |f|^2 e^{-\varphi}=\int_{\mathbb B} |h|^2 e^{-\varphi_\varepsilon}\le {\rm const}_n.
$$
\end{proof}

\begin{proof}[Proof of Theorem 1.3]
Without loss of generality, we may assume $a=0$ and $|z|<e^{-1}$ on $\Omega$. For small $\varepsilon>0$, we put
$$
\phi(z)=\log (|z|^2+\varepsilon^2),\ \ \ \eta=-\phi+\log(-\phi).
$$
Since
$$
-\partial\bar{\partial}\eta = \partial\bar{\partial}\phi-\phi^{-1}\partial\bar{\partial}\phi+\phi^{-2}\partial\phi\wedge \bar{\partial}\phi,\ \ \ -\partial\eta=(1-\phi^{-1})\partial \phi,
$$
we conclude that
\begin{equation}
-i \partial\bar{\partial}\eta-\lambda^{-1}i\partial\eta\wedge \bar{\partial}\eta\ge i\partial\bar{\partial}\phi
\end{equation}
where $\lambda=(1-\phi)^2$. Let $f$ be the holomorphic function in Lemma 3.3. Put $\Theta=i\partial\bar{\partial}(\phi+|z|^2)$ and $v=f\bar{\partial}\chi(|z|^2/\varepsilon^2)\wedge dz_1\wedge\cdots\wedge dz_n$. If $\varepsilon$ is sufficiently small, then $v\in D_{(n,1)}(\Omega)$ satisfies
\begin{eqnarray*}
([\Theta,\Lambda]^{-1}v,v)_{\varphi+|z|^2+2n\log|z|} & \le & \sup|\chi'|^2 \int_{B_\epsilon\backslash B_{\varepsilon/\sqrt{2}}} |f|^2\frac{|z|^2}{\varepsilon^4}\frac{(|z|^2+\varepsilon^2)^2}{\varepsilon^2}|z|^{-2n}e^{-\varphi-|z|^2}\\
& \le & {\rm const}_n \varepsilon^{-2n} \int_{B_\varepsilon}|f|^2 e^{-\varphi}\le {\rm const}_n.
\end{eqnarray*}
By Proposition 3.2, there is a solution $u_\varepsilon$ of the equation $\bar{\partial}u=v$ such that
$$
\int_\Omega \frac{|u^\ast_\varepsilon|^2}{(\sqrt{\eta}+\sqrt{\lambda})^2} e^{-\varphi-|z|^2-2n\log|z|}\le {\rm const}_n
$$
where $u=u^\ast dz_1\wedge\cdots\wedge dz_n$. Put $f_\varepsilon=f\chi(|z|^2/\varepsilon^2)-u_\varepsilon$. Then we have $f_\varepsilon \in {\mathcal O}(\Omega)$, $f_\varepsilon(0)=f(0)=c$, and
$$
\int_\Omega \frac{|f_\varepsilon|^2}{|z|^{2n}(\log|z|)^2} e^{-\varphi}\le {\rm const}_n,
$$
for $\sqrt{\eta}+\sqrt{\lambda}\asymp |\log(|z|^2+\varepsilon^2)|$.
To complete the proof, it suffices to take a limit of $\{f_\varepsilon\}$ as $\varepsilon\rightarrow 0$.
\end{proof}

\section{Proof of Theorem 1.2}

 It suffices to show that $\varphi$ is locally bounded above near $E$. Replacing ${\mathbb D}^{n+1}\backslash E$ by ${\mathbb D}^{n+1}_r\backslash E$ for $0<r<1$, we may assume that $E$ is a closed complete pluripolar subset of a neighborhood of $\overline{D}^{n+1}$. Thanks to Demailly \cite{DemaillyBook}, Chapter 3, Lemma 2.2, there is a function
 $$
 \rho\in PSH^{-}({\mathbb D}^{n+1})\cap C^\infty({\mathbb D}^{n+1}\backslash E)
 $$
  such that $\rho=-\infty$ on $E\cap {\mathbb D}^{n+1}$. Here $PSH^-$ stands for the set of negative psh functions. Thus
   $$
   -{\partial}\bar{\partial}\log(-\rho)-\sum_{j=1}^{n+1}{\partial}\bar{\partial}\log(1-|z_j|)
   $$
   defines a complete K\"ahler metric on ${\mathbb D}^{n+1}\backslash E$.

 Let $B\subset\subset {\mathbb D}^n$ be a ball and $z^0\in (B\times {\mathbb D}_{r})\backslash E$. Applying Theorem 1.3 in a similar way as Demailly \cite{Demailly92}, we get  a holomorphic function $f$ on ${\mathbb D}^{n+1}\backslash E$ with the following properties: $f(z^0)=e^{\varphi(z^0)/2}$ and
 $$
 \int_{{\mathbb D}^{n+1}\backslash E} |f|^2 e^{-\varphi}\le C_n.
 $$
By Fubini's theorem, there exists a set $Z_1\subset {\mathbb D}^n$ of zero Lebesgue measure such that
  $$
 \int_{{\mathbb D}}|f(z',\cdot)|^2 e^{-\varphi(z',\cdot)}<\infty.
 $$
 for any $z'\in {\mathbb D}^n\backslash Z_1$. It is easy to see that there is a set $Z_2\subset {\mathbb D}^n$ of zero Lebesgue measure such that $(\{z'\}\times {\mathbb D})\cap E$ is a closed polar set in ${\mathbb D}$ for any $z'\in {\mathbb D}^n\backslash Z_2$. Thus $A-Z_1-Z_2$ is of positive Lebesgue measure and for any $z'\in A-Z_1-Z_2$, $f(z',\cdot)$ is locally $L^2$ near the closed polar set $(\{z'\}\times {\mathbb D})\cap E$, thus it extends holomorphically across this polar set, in view of Theorem 2.1. Observe that $f\in {\mathcal O}({\mathbb D}^n\times A(r,1))$ where $A(r_1,r_2)$ stands for an annulus, and for any fixed $z'$ in the set $A-Z_1-Z_2$ (which is of positive measure), $f(z',\cdot)$ extends to a holomorphic function on ${\mathbb D}$. By using Laurent series expansion of $f$ in $z_{n+1}$, we conclude that $f$ extends holomorphically to ${\mathbb D}^{n+1}$ (compare \cite{HarveyPolking75}, Theorem 2.1).

 Now choose $r<r''<r'<1$ and ball $B'$ with $B\subset\subset B'\subset \subset {\mathbb D}^n$. By Cauchy's integrals, we have
 \begin{eqnarray*}
 |f(z^0)|^2  & \le & C \int_{B'\times ({{\mathbb D}_{r'}\backslash {\mathbb D}_{r''}})}|f|^2\\
 & \le & C \sup_{B'\times ({{\mathbb D}_{r'}\backslash {\mathbb D}_{r''}})}e^\varphi \int_{B'\times ({{\mathbb D}_{r'}\backslash {\mathbb D}_{r''}})}|f|^2e^{-\varphi}\\
 & \le & C \sup_{B'\times ({{\mathbb D}_{r'}\backslash {\mathbb D}_{r''}})}e^\varphi
 \end{eqnarray*}
 where $C$ is a generic constant depending only on $B,B',r,r',r''$. Since $f(z^0)=e^{\varphi(z^0)/2}$, we obtain
 $$
 \varphi(z^0)\le \log C + \sup_{B'\times ({{\mathbb D}_{r'}\backslash {\mathbb D}_{r''}})}\varphi.
 $$
 Since $B$ and $z^0$ are arbitrarily chosen, we conclude the proof.

 \begin{conjecture}
Let $E$ be a closed pluripolar subset of ${\mathbb D}^{n+1}$ such that $E\subset {\mathbb D}^{n}\times {\mathbb D}_r$ for some $0<r<1$. Let $\varphi\in PSH({\mathbb D}^{n+1}\backslash E)$.  Suppose
    there is a nonpluripolar subset $A$ of ${\mathbb D}^{n}$ such that for each $z'\in A$, the restriction of $\varphi$ to $(\{z'\}\times {\mathbb D})\backslash E$ is locally bounded from above near $(\{z'\}\times {\mathbb D})\cap E$.  Then $\varphi$
 can be extended to a psh function on ${\mathbb D}^{n+1}$.
\end{conjecture}

 \section{Proof of Theorem 1.4}
The idea is due to Siu \cite{SiuLelong}. The first step is to show the following

 \begin{proposition}
 Let $E$ be a compact polar subset of\/ ${\mathbb D}_r$ for some $0<r<1$. Put
 \begin{eqnarray*}
\Omega & = & ({\mathbb D}_\delta^n\times {\mathbb D})\cup ({\mathbb D}^{n}\times ({\mathbb D}\backslash E)) \\
G & = & ({\mathbb D}_\varepsilon^n\times {\mathbb D})\cup ({\mathbb D}^{n}\times ({\mathbb D}-\overline{\mathbb D}_s))
\end{eqnarray*}
for $0<\varepsilon<\delta<1$ and $0<r<s<1$. Let $(L,h)$ be a semipositive holomorphic line bundle over $\Omega$. Let $h$ be given locally by $e^{-\varphi}$ such that $e^{-\varphi}\in L^1_{\rm loc}(G)$. Then $L$ can be extended to a holomorphic line bundle over ${\mathbb D}^{n+1}$.
\end{proposition}

\begin{proof} Take a covering $\{U_1,U_2\}$ of $\Omega$ where
$$
U_1  =  {\mathbb D}^{n}\times ({\mathbb D}\backslash E),\ \ \
U_2 =  {\mathbb D}_\delta^n\times {\mathbb D}.
$$
 Clearly, $L$ is trivial over each of them. Let $\Gamma(\Omega,L)$ denote the space of holomorphic sections of $L$ over $\Omega$. We are going to show that $\Gamma(\Omega,L)$ generates all fibres $L_z$, $z\in G$. Granted this, the conclusion follows from the theory of coherent subsheaf extension (compare \cite{SiuLelong}, p. 146).

Given $z^0\in U_1\cap G$. Since $U_1$ is pseudoconvex and $h=e^{-\varphi}\in L^1_{\rm loc}(U_1\cap G)$ with $\varphi\in PSH(U_1)$,  there exists $s\in \Gamma(U_1,L)$ such that $s(z^0)\neq 0$ and $\int_{U_1}|s|_h^2 <\infty$, in view of the H\"ormander-Bombieri-Skoda theorem. Since $L$ is trivial over $U_2$, there exist a global frame $\zeta$ on $U_2$ and a function $\varphi\in PSH(U_2)$ such that $h=e^{-\varphi}$. Write $s=f\otimes \zeta$ with $f\in {\mathcal O}(U_2\backslash E)$. As
$$
\int_{U_{2}\backslash E}|f|^2e^{-\varphi}\le \int_{U_1}|s|_h^2 <\infty
$$
and $\varphi$ is locally bounded above near $E$, it follows that $f$ can be extended to a holomorphic function on $U_2$, in view of Corollary 2.4. That is, $s$ can be extended to a holomorphic section of $L$ over $\Omega$ which generates the fibre $L_{z^0}$.

Now let $z^0\in U_2\cap G\backslash U_1$. There is $f\in {\mathcal O}(U_2)$ such that $f(z^0)\neq 0$ and $\int_{U_2}|f|^2e^{-\varphi}<\infty$. Put $s=f\otimes \zeta$. A standard application of H\"ormander's $L^2$ estimates for the $\bar{\partial}-$operator yields a section $s_1\in \Gamma(U_1,L)$ such that $s_1=s$ on
 $$
 U_1\cap \{z_1=z^0_1,\cdots,z_n=z^0_n\},
 $$
 and
 $$
 \int_{U_1}|s_1|^2_h \le {\rm const}_\delta \int_{U_2} |s|^2_h <\infty.
 $$
 Similar as above, $s_1$ may be extended to a holomorphic section $s_2$ of $L$ over $\Omega$. Since $s_2(z^0)=s(z^0)\neq 0$, it follows that $s_2\in \Gamma(\Omega,L)$ generates the fibre $L_{z^0}$.
\end{proof}

\begin{proof}[Proof of Theorem 1.4]
Replacing $\Omega$ by
$$
({\mathbb D}^n_\delta\times {\mathbb D}_r)\cup \left({\mathbb D}^n_r\times ({\mathbb D}_r\backslash E)\right)
$$
for $0<r<1$, and multiplying $u$ by a sufficiently small positive number, we can assume without loss of generality that the Lelong number $n(u,z)$ of $u$ at $z$ is less than $2$ for
$$
z\in G:=({\mathbb D}_\varepsilon^n\times {\mathbb D})\cup ({\mathbb D}^{n}\times ({\mathbb D}-\overline{\mathbb D}_s))
$$
where $0<\varepsilon<\delta$ and $r<s<1$. Choose an open covering $\{U_j\}$ of $\Omega$ consisting of polydiscs, and $\varphi_j\in PSH(U_j)$ such that $u=\frac{i}{\pi}\partial\bar{\partial}\varphi_j$ on $U_j$. By a theorem of Skoda \cite{Skoda72}, we have $e^{-\varphi_j}\in L^1_{\rm loc}(G\cap U_j)$. Put $U_{jk}=U_j\cap U_k$. We may choose $f_{jk}\in {\mathcal O}(U_{jk})$ skew-symmetric in $j,k$ such that $2{\rm Re\,}f_{jk}=\varphi_j-\varphi_k$.
 Since $f_{jk}+f_{kl}+f_{lj}$ is a purely imaginary constant on $U_{jkl}:=U_j\cap U_k\cap U_l$ and $H^2(\Omega,{\mathbb R})\cong{\rm Hom}(H_2^{{\rm sing}}(\Omega,\mathbb Z),\mathbb R)=0$ (note that every $2-$cycle in $\Omega$ is homologous to a $2-$cycle in the {\it contractible}\/ domain ${\mathbb D}^n_\delta\times {\mathbb D}$ via the homotopy given by $(z',z_{n+1})\mapsto ((\delta+t(1-\delta))z',z_{n+1})$, $t\in [0,1]$, thus $H_2^{{\rm sing}}(\Omega,\mathbb Z)=0$), we conclude that there exists $c_{jk}\in {\mathbb R}$ skew-symmetric in $j,k$ such that
 $$
 f_{jk}+f_{kl}+f_{lj}=i(c_{jk}+c_{kl}+c_{lj})
 $$
 on $U_{jkl}$. Put $g_{jk}=\exp(f_{jk}-ic_{jk})$. Then there is a holomorphic line bundle $L$ over $\Omega$ with transition functions $\{g_{jk}\}$ for the covering $\{U_j\}$.

   Thanks to Proposition 5.1, we conclude that $L$ is globally trivial over $\Omega$. Hence there are nowhere zero holomorphic functions $g_j$ on $U_j$ such that $g_{jk}=g_kg_j^{-1}$ on $U_{jk}$. Let $\varphi$ be the psh function on $\Omega$ which equals $\varphi_j+2\log |g_j|$ on $U_j$. By Theorem 1.2, $\varphi$ can be extended to a psh function on ${\mathbb D}^{n+1}$, so that the closed positive $(1,1)-$current $\frac{i}\pi \partial\bar{\partial}\varphi$ on ${\mathbb D}^{n+1}$ extends $u$.
\end{proof}

\section{Domains on which the conclusion of Theorem 1.3 fails}

A generalized annulus is a domain defined as ${\mathcal A}:=D\backslash \overline{D_0}$, where $D_0\subset \subset D$ are two smooth bounded domains in ${\mathbb C}^n$. We are going to show

\begin{proposition}
If $n\ge 2$, then the conclusion of Theorem 1.3 fails for ${\mathcal A}$.
\end{proposition}

\begin{proof}
Suppose on the contrary Theorem 1.3 holds on ${\mathcal A}$. We are going to show that every smooth psh function on ${\mathcal A}$ can be extended to a psh function on $D$. This would contradict with a theorem of Bedford-Taylor \cite{BedfordTaylor88} that there always exists a smooth psh function on ${\mathcal A}$ which does not admit a psh extension to any larger domain.

Now fix a smooth psh function $\varphi$ on ${\mathcal A}$. For each positive integer $m$, we define $H^2({\mathcal A},m\varphi)$ to be the Hilbert space of holomorphic functions $f$ on ${\mathcal A}$ such that
$$
\int_{\mathcal A} |f|^2 e^{-m\varphi}<\infty.
$$
Let $K_{m\varphi}$ denote the related Bergman kernel. Since Theorem 1.3 holds on ${\mathcal A}$, we may prove similarly as Demailly \cite{Demailly92} the following
\begin{equation}
\varphi(z)-\frac{C_1}m \le \frac1m \log K_{m\varphi}(z)\le \sup_{|\zeta-z|<r} \varphi(\zeta)+\frac1m \log \frac{C_2}{r^n}
\end{equation}
for every $z\in {\mathcal A}$ and $r<d(z,\partial {\mathcal A})$, where $C_1,C_2$ are constants independent of $m$. Fix two arbitrary domains $D_1,D_2$ such that $D_0\subset\subset D_1\subset\subset D_2\subset D$. Since $\varphi$ is smooth on ${\mathcal A}$, so we may choose $r=1/m$ in (6.1) such that
\begin{equation}
\left|\frac1m \log K_{m\varphi}(z)-\varphi(z)\right|\le \frac{C}m\log m
\end{equation}
holds uniformly on $D_2\backslash D_1$ for sufficiently large $m$. Let $\{f_{m,k}\}_{k\ge 1}$ be an orthonormal basis of $H^2({\mathcal A},m\varphi)$. By Hartogs' theorem, every $f_{m,k}$ may be extended to a holomorphic function on $D$.  It follows that
$$
\frac1m \log K_{m\varphi}=\frac1m \log \sum_k |f_{m,k}|^2,\ \ \ m\ge 1,
$$
can be extended to psh functions on $D$, which are locally uniformly bounded from above in $D$ in view of (6.2) and the maximum principle. Thus
$$
\varphi_j:=\left(\sup_{m\ge j}\left\{\frac1m \log K_{m\varphi}\right\}\right)^\ast,\ \ \ j\ge 1,
$$
are psh functions on $D$. Here $u^\ast$ stands for the upper semicontinuous regularization of $u$. Since $\{\varphi_j\}_{j\ge 1}$ is a decreasing sequence of psh function on $D$, so the limit function $\tilde{\varphi}$ is also psh on $D$. On the other hand, we have $\varphi_j\rightarrow \varphi$ pointwise on ${\mathcal A}$ in view of (6.2), so that $\tilde{\varphi}=\varphi$ on ${\mathcal A}$.
\end{proof}

\end{document}